\documentclass[12pt]{article}
\usepackage{amsmath, amsthm}
\usepackage{amscd}
\usepackage{amssymb}
\usepackage{array}
\usepackage{color}

\newtheorem{thm}{Theorem}[section]
\newtheorem{theorem}[thm]{Theorem}

\newtheorem{lemma}[thm]{Lemma}
\newtheorem{proposition}[thm]{Proposition}

\theoremstyle{definition}
\newtheorem{definition}[thm]{Definition}
\newtheorem{example}[thm]{Example}

\newtheorem{observation}[thm]{Observation}

\begin{document}


\newcommand{\green}[1]{{\textbf{#1}}}
\newcommand{\Red}[1]{{\color{red}{#1}}}

\newcommand{\id}{\relax{\rm 1\kern-.28em 1}}
\newcommand{\R}{\mathbb{R}}
\newcommand{\C}{\mathbb{C}}
\newcommand{\Cbar}{\overline{\C}}
\newcommand{\Z}{\mathbb{Z}}
\newcommand{\Q}{\mathbb{Q}}
\newcommand{\kk}{k}
\newcommand{\bD}{\mathbb{D}}
\newcommand{\bG}{\mathbb{G}}
\newcommand{\bP}{\mathbb{P}}
\newcommand{\bM}{\mathbb{M}}
\newcommand{\g}{\mathfrak{G}}
\newcommand{\fh}{\mathfrak{h}}
\newcommand{\e}{\epsilon}
\newcommand{\Uu}{\mathfrak{U}}

\newcommand{\cA}{\mathcal{A}}
\newcommand{\cB}{\mathcal{B}}
\newcommand{\cC}{\mathcal{C}}
\newcommand{\cD}{\mathcal{D}}
\newcommand{\cI}{\mathcal{I}}
\newcommand{\cL}{\mathcal{L}}
\newcommand{\cK}{\mathcal{K}}
\newcommand{\cO}{\mathcal{O}}
\newcommand{\cG}{\mathcal{G}}
\newcommand{\cJ}{\mathcal{J}}
\newcommand{\cF}{\mathcal{F}}
\newcommand{\cP}{\mathcal{P}}
\newcommand{\cU}{\mathcal{U}}
\newcommand{\ep}{\mathcal{E}}
\newcommand{\E}{\mathcal{E}}
\newcommand{\cH}{\mathcal{O}}
\newcommand{\cV}{\mathcal{V}}
\newcommand{\cPO}{\mathcal{PO}}
\newcommand{\cHol}{\mathrm{H}}

\newcommand{\rGL}{\mathrm{GL}}
\newcommand{\rSU}{\mathrm{SU}}
\newcommand{\rSL}{\mathrm{SL}}
\newcommand{\rPSL}{\mathrm{PSL}}
\newcommand{\rSO}{\mathrm{SO}}
\newcommand{\rOSp}{\mathrm{OSp}}
\newcommand{\rSpin}{\mathrm{Spin}}
\newcommand{\rsl}{\mathrm{sl}}
\newcommand{\rsu}{\mathrm{su}}
\newcommand{\rM}{\mathrm{M}}
\newcommand{\rdiag}{\mathrm{diag}}
\newcommand{\rP}{\mathrm{P}}
\newcommand{\rdeg}{\mathrm{deg}}
\newcommand{\pt}{\mathrm{pt}}
\newcommand{\red}{\mathrm{red}}

\newcommand{\bm}{\mathbf{m}}

\newcommand{\M}{\mathrm{M}}
\newcommand{\End}{\mathrm{End}}
\newcommand{\Hom}{\mathrm{Hom}}
\newcommand{\diag}{\mathrm{diag}}
\newcommand{\rspan}{\mathrm{span}}
\newcommand{\rank}{\mathrm{rank}}
\newcommand{\Gr}{\mathrm{Gr}}
\newcommand{\ber}{\mathrm{Ber}}

\newcommand{\str}{\mathrm{str}}
\newcommand{\Sym}{\mathrm{Sym}}
\newcommand{\tr}{\mathrm{tr}}
\newcommand{\defi}{\mathrm{def}}
\newcommand{\Ber}{\mathrm{Ber}}
\newcommand{\spec}{\mathrm{Spec}}
\newcommand{\sschemes}{\mathrm{(sschemes)}}
\newcommand{\sschemeaff}{\mathrm{ {( {sschemes}_{\mathrm{aff}} )} }}
\newcommand{\rings}{\mathrm{(rings)}}
\newcommand{\Top}{\mathrm{Top}}
\newcommand{\sarf}{ \mathrm{ {( {salg}_{rf} )} }}
\newcommand{\arf}{\mathrm{ {( {alg}_{rf} )} }}
\newcommand{\odd}{\mathrm{odd}}
\newcommand{\alg}{\mathrm{(alg)}}
\newcommand{\sa}{\mathrm{(salg)}}
\newcommand{\sets}{\mathrm{(sets)}}
\newcommand{\smflds}{\mathrm{(smflds)}}
\newcommand{\mflds}{\mathrm{(mflds)}}
\newcommand{\shcps}{\mathrm{(shcps)}}
\newcommand{\sgrps}{\mathrm{(sgrps)}}
\newcommand{\SA}{\mathrm{(salg)}}
\newcommand{\salg}{\mathrm{(salg)}}
\newcommand{\varaff}{ \mathrm{ {( {var}_{\mathrm{aff}} )} } }
\newcommand{\svaraff}{\mathrm{ {( {svar}_{\mathrm{aff}} )}  }}
\newcommand{\ad}{\mathrm{ad}}
\newcommand{\Ad}{\mathrm{Ad}}
\newcommand{\pol}{\mathrm{Pol}}
\newcommand{\Lie}{\mathrm{Lie}}
\newcommand{\Proj}{\mathrm{Proj}}
\newcommand{\rGr}{\mathrm{Gr}}
\newcommand{\rFl}{\mathrm{Fl}}
\newcommand{\rPol}{\mathrm{Pol}}
\newcommand{\rdef}{\mathrm{def}}
\newcommand{\ah}{\mathrm{ah}}

\newcommand{\uspec}{\underline{\mathrm{Spec}}}
\newcommand{\uproj}{\mathrm{\underline{Proj}}}

\newcommand{\sym}{\cong}

\newcommand{\al}{\alpha}
\newcommand{\lam}{\lambda}
\newcommand{\de}{\delta}
\newcommand{\ttau}{\tilde \tau}
\newcommand{\D}{\Delta}
\newcommand{\s}{\sigma}
\newcommand{\lra}{\longrightarrow}
\newcommand{\ga}{\gamma}
\newcommand{\ra}{\rightarrow}

\newcommand{\wbar}{\overline{w}}
\newcommand{\zbar}{\overline{z}}
\newcommand{\fbar}{\overline{f}}
\newcommand{\etabar}{\overline{\eta}}
\newcommand{\zetabar}{\overline{\zeta}}
\newcommand{\betabar}{\overline{\beta}}
\newcommand{\albar}{\overline{\alpha}}
\newcommand{\abar}{\overline{a}}
\newcommand{\dbar}{\overline{d}}
\newcommand{\tbar}{\overline{t}}
\newcommand{\thetabar}{\overline{\theta}}
\newcommand{\Mbar}{{\overline{M}}}
\newcommand{\aubar}{\underline{a}}
\newcommand{\gabar}{\overline{\gamma}}
\newcommand{\fg}{\mathfrak{g}}
\newcommand{\Span}{\mathrm{span}}

\newcommand{\NOTE}{\bigskip\hrule\medskip}

\newcommand{\G}{{(\C^{1|1})}^\times}
\newcommand{\pair}[2]{\langle \, #1, #2\, \rangle}
\newcommand{\cinfty}{\mathcal{C}^\infty}

\medskip

\centerline{\Large \bf   The Peter-Weyl Theorem for $\rSU(1|1)$\\} 

\medskip

\bigskip

\centerline{C. Carmeli$^\natural$, R. Fioresi$^\flat$, S. Kwok$^\star$}

\medskip

\centerline{\it $^\natural$ DIME, Universit\`a di Genova Via Magliotto 2, I-17100 Savona, Italy}
 \centerline{\it  Via Magliotto 2, I-17100 Savona, Italy.}
\centerline{{\footnotesize e-mail: claudio.carmeli@gmail.com}}

\medskip

\centerline{\it $^\flat$ Dipartimento di Matematica, Universit\`{a} di
Bologna }
 \centerline{\it Piazza di Porta S. Donato, 5. 40126 Bologna. Italy.}
\centerline{{\footnotesize e-mail: rita.fioresi@UniBo.it}}

\medskip

\centerline{\it $^\star$ Mathematics Research Unit, University of Luxembourg}
 \centerline{\it 6, Rue Richard Coudenhove-Kalergi, L-1359, Luxembourg.}
\centerline{{\footnotesize e-mail: 
stephen.kwok@uni.lu}}

\bigskip

\begin{abstract}
We study a generalization of the results in \cite{cfk} to the case of
$\rSU(1|1)$ interpreted as the supercircle $S^{1|2}$. We describe all of its
finite dimensional complex irreducible representations, 
we give a reducibility result for representations not containing the trivial character, and we compute explicitly the corresponding
matrix elements. In the end we give the  Peter-Weyl theorem for  $S^{1|2}$.
\end{abstract}

\section{Introduction} 
\label{intro}

The theory of representations of compact supergroups has not yet been
fully understood and in particular there is neither a decisive classification result
in this category, nor a thorough treatment of the fundamental results, as
for example the Peter-Weyl theorem. 
In this paper we want to proceed and give another important example, beside the
one already studied in \cite{cfk}, namely we want to fully discuss the case
of $\rSU(1|1)$. This is a natural generalization of the $S^{1|1}$ case
studied in \cite{cfk} and \cite{kwok}: 
the reduced group is still $S^1$ in both cases, 
but here we are considering two odd variables, so
we may very well call $\rSU(1|1)$, $S^{1|2}$ the supercircle
with two odd dimensions. This
generalization is non trivial, because it is well known that 
as soon as the odd dimension becomes greater than one, new supersymmetric
phenomenona may appear to make the whole theory diverge significantly
from the ordinary one, 
though in this particular
case, it does not happen. For example in \cite{fk2}, we see that,
$\mathrm{Aut}(\bP^{1|1})$,  the 
automorphism group of the projective superspace in one odd dimension 
coincides with the projective linear supergroup, but this
isomorphism is lost when the odd dimension is greater than one. In
case of odd dimension one, there is
a strong connection between $\mathrm{Aut}(\bP^{1|1})$ 
and the theory of SUSY curves
(see also \cite{dw1}, \cite{fk}).
In \cite{cfk}, the theory of SUSY curves was linked with real forms of
$(\C^{1|1})^\times$ and it was proven the remarkable result that the
real forms reducing to $S^1$ are actually all isomorphic to $S^1$.  
When the odd dimension is greater than one, the theory departs significantly
from the ordinary one, the automorphism group of $\bP^{1|n}$ is not
the projective linear supergroup and the connection with SUSY curves 
and $S^{1|n}$
is either lost or not immediately evident. 

\medskip
This is our main motivation for examining the case $\rSU(1|1)$: we want
to see if new phenomena arise and if, in the future, we can try to establish
a connection with SUSY curves and shed light on this part of mathematics
still very actively studied (see \cite{dw1} and \cite{witten}).

\medskip
Our paper is organized as follows. In Section \ref{prelim} we recall
very briefly our notation and few facts about real forms both in
the sheaf theoretic and Super Harish-Chandra pair approach to supergroups.
In Section \ref{s11-sec} we completely classify the 
finite-dimensional irreducible complex representations of
the supergroup $S^{1|1}$, and we also give a reducibility result for representations not containing the trivial character. 
This section completes the study initiated in
\cite{cfk}. In Section \ref{su-sec} we classify the representations
of $\rSU(1|1)$ and we prove the super version of the Peter-Weyl Theorem.

\medskip
{\bf Acknowledgements}. 
All of the authors want to express their
deep gratitude to Prof. V. S. Varadarajan for his constant encouragement
and all of his suggestions, which have led
not just to this paper, but to many results the authors proved 
together or alone throughout
their careers.

\section{Supergeometry preliminaries} \label{prelim}

We would like to quickly summarize few definitions and key facts about
supergeometry especially to establish our notation; for all the details
see \cite{ccf}, \cite{dm}, \cite{vsv2} and 
\cite{fl}, besides our main reference \cite{cfk}.

\medskip
\noindent
Let us take our ground field $k=\R$ or $\C$.

\begin{definition} \label{Tpt}
A {\it superspace} $S=(|S|, \cO_S)$ is a topological space $|S|$
together with a sheaf of superalgebras $\cO_S$, such that the stalk $\cO_{S,x}$
is a local superalgebra, for $x \in |S|$.
A {\it morphism} $\phi:S \lra T$ of superspaces is given by
$\phi=(|\phi|, \phi^*)$, where $\phi: |S| \lra |T|$ is a map of
topological spaces and $\phi^*:\cO_T \lra \phi_*\cO_S$ is
a local sheaf morphism.
A  \textit{differentiable (resp. analytic) supermanifold}
of dimension $p|q$ is a superspace $M=(|M|, \cO_M)$
where $|M|$ is manifold and $\cO_M$ is a sheaf of superalgebras over 
$\R$ (resp. $\C$),
which is locally isomorphic  to 
$\R^{p|q}$ (resp. $\C^{p|q}$), where $\R^{p|q}=(\R^p, C^\infty_{\R^p}\otimes
\wedge(\xi_1, \dots, \xi_q))$ and similarly for $\C^{p|q}$.

\medskip
Let $M$ and $T$ be supermanifolds.  
A \textit{$T$-point} of a supermanifold $M$ is 
a morphism $T \longrightarrow M$ ($T \in \smflds$).  We denote the set of 
all $T$-points 
by $M(T)$. We define the \textit{functor of points} of $M$:
\begin{equation} \label{fopts-eq}
M: \text{(smflds)}^o \lra \sets, \quad T \mapsto M(T), \quad
M(\phi)(\psi)=\psi \circ \phi,
\end{equation}
where $\text{(smflds)}$ denotes the category of supermanifolds and
the index $o$ as usual refers to the opposite category. 
$\psi$ here is a $T$-point of $M$, i.e. $\psi: T \lra M$. We
shall write $\text{(smflds)}_\R$ or $\text{(smflds)}_\C$ whenever it is
necessary to distinguish between real or complex supermanifolds.
\end{definition}

We want to define the real supermanifold underlying a complex supermanifold
and the concept of real form.

\begin{definition} \label{complexconj}
Let $M=(|M|, \cO_M)$ be a complex supermanifold. We define 
the \textit{complex conjugate} $\Mbar$ of $M$ as the complex supermanifold
$\Mbar=(|M|, \cO_{\Mbar})$, where $\cO_\Mbar$ is $\cO_M$ with the 
$\C$-antilinear structure. 
We shall denote the map realizing the
$\C$-antilinear isomorphism between $M$ and $\Mbar$ as
$\sigma: M \lra \Mbar$ and sometimes we shall write $(\sigma^*)^{-1}(f)=\fbar$,
where $f$ is a section of the sheaf $\cO_M$.
We define a \textit{real
structure} on $M$ as an involutive isomorphism of ringed spaces
$\rho: M \lra  \Mbar$, which is $\C$-linear on the sheaves, and such that
$\rho^*: \cO_{\Mbar} \lra \rho_*\cO_M$ is a $\C$-linear sheaf isomorphism.
We furtherly define the isomorphism of ringed superspaces
$\psi=\sigma^{-1} \circ \rho: M \lra M$, which is
$\C$-antilinear on the sheaves $\psi^*=\rho^*\circ (\sigma^*)^{-1}:
\cO_M \lra  \rho_*\cO_M$.

\medskip
The superspace $M^\rho=(|M^\rho|, \cO_{M^\rho})$, 
where $|M|^{|\rho|}$ consists of 
the fixed points of $\rho$ and $\cO_{M^\rho}$ consists of sections
$f\in \cO_M|_{M^\rho}$ such that $\psi^*(f)=f$,
is the  \textit{real form} of $M$ defined by $\rho$. 

\medskip
If $G$ is a complex supergroup and $\rho$ is a supergroup morphism, then
$G^\rho$ is a real supergroup.
\end{definition}

Through the notion of real form it is possible to define
the concept of real underlying supermanifold, which is mostly important for us.

\begin{definition} \label{realforms1} 
We define on $M \times \Mbar$ the real
structure $\tau:  M \times \Mbar \lra \Mbar \times M$, 
$|\tau|(x,y)=(y,x)$ 
and $\tau^*: \cO_{\Mbar \times M} \lra \tau_*\cO_{M \times \Mbar}$ given by:
\begin{equation} \label{tau-eq}
\tau(f \otimes g)=(-1)^{|f||g|}g \otimes f
\qquad f \in \cO_{\Mbar}, \, g \in \cO_M.
\end{equation}
We call $M^\tau$ the \textit{real supermanifold underlying} $M$
and we denote it with $M_\R$.
\end{definition}

The next example is very instructive and essential for our
treatment.

\begin{example}\label{realform-ex11}
We want to understand $\C^{m|n}_\R$ 
the real supermanifold underlying $\C^{m|n}$. 
We define $|\tau|: |\C^{m|n}| \times |\overline{\C}^{m|n}| \lra
|\overline{\C}^{m|n}| \times |\C^{m|n}|$ as $|\tau|(p,q)=(q,p)$, while 
on the sheaves we define  
the $\C$-linear isomorphism 
$\tau^*: \cO_{\Cbar^{m|n} \times \C^{m|n}}$ $\lra$ $\tau_* \cO_{\C^{m|n} \times
\Cbar^{m|n}}$,
$\tau^*(w_i)=z_i$, $\tau(\eta_j)=\zeta_j$,
$\tau^*(z_i)=w_i$, $\tau(\zeta_j)=\eta_j$,
where $(z_i,\zeta_j)$ and $(w_i, \eta_j)$ are global coordinates
on $\C^{m|n}$ and $\Cbar^{m|n}$ respectively (with
a common abuse of notation, we write $z_i \in \cO_{\Cbar^{m|n} \times \C^{m|n}}$ 
in place of the more appropriate 
$1 \otimes z_i$ and similarly for the rest of the
coordinates).
We associate to $\tau^*$ (see (\ref{tau-eq}))
the $\C$-antilinear isomorphism
$\psi^*=\tau^* \circ ((\sigma^*)^{-1} \otimes \sigma^*)$:
\begin{equation} \label{psi-eq}
\psi^*(w_i)=\zbar_i,\, \psi^*(\eta_j)=\zetabar_j, \, 
\psi(z_i)=\wbar_i,\, \psi(\zeta_j)=\etabar_j,
\end{equation}
where we write $\zbar_i$ instead of $(\sigma^*)^{-1}(z_i)$ etc.
We can then write immediately global coordinates on $\C^{m|n}_\R$:
$$
x_i=(z_i+\zbar_i)/2, \quad y_i=(z_i-\zbar_i)/2i, \quad 
\mu_j=(\zeta_j+\zetabar_j)/2, \quad
\nu_j=(\zeta_j-\zetabar_j)/2i
$$
As for the $T$-points, $T \in \smflds_\R$, we have:
$$
\begin{array}{rl}
\C^{m|n}_\R(T)&=\Hom_{\smflds_\R}(T, \C^{m|n}_\R)=\Hom_{\salg_\R}
 (\cO(\C^{m|n}_\R), \cO(T))= \\ \\
&=\{ \phi:  \cO(\C^{m|n}_\R)
\lra \cO(T)\}= \\ \\
&=\{(t_0^1,t_1^1, \dots,t_0^m,t_1^m,  \theta_0^1, \theta_1^1, \dots,
\theta_0^n, \theta_1^n) \, | \\ \\ &\, \qquad
t_0^k,t_1^k  \in \cO(T)_0, \, \theta_0^j,\theta_1^j \in \cO(T)_1\}
\end{array}
$$
Evidently $\C^{m|n}_\R=\R^{2m|2n}$ as one expects.

\medskip
It is customary to define: 
$$
t^k:=t_0^k+it_1^k \quad \tbar^k:=t_0^k-it_1^k, \quad 
\theta^j:=\theta_0^j+i\theta_1^j, \quad \thetabar^j:=\theta_0^j-i\theta_1^j
$$
These are elements in $\C^{m|n}_\R(T) \otimes \C$.

\medskip
Using the language of functor of points and the
local coordinates $t^i$, $\tbar^i$, $\theta^j$, $\thetabar^j$,
it is then very easy to give 
a real form of a given supermanifold, by giving a ($\C$-antilinear)
involution of $\C^{m|n}_\R(T) \otimes \C$ functorial in $T$.
For example we can define:
$$
\sigma(t^k)=\tbar^k, \qquad \sigma(\theta^j)=\thetabar^j
$$
(notice that since $\sigma$ is an involution, it is enough to
define the image of $t^k$ and $\theta^j$). We have immediately that
$(\C^{m|n}_\R(T) \otimes \C)^\sigma=\R^{m|n}(T)$, so that
$\R^{m|n}(T)$ is the real form of $\C^{m|n}_\R(T)$ corresponding to
the involution $\sigma$.
There are however deep questions
on the foundation of the theory regarding the apparent simplicity of
this definition.
We invite the reader to consult \cite{cfk}.
For the case of $\C^{m|n}_\R$ and its open sets, which
is the only case regarding the present work, we do not need the
full theory developed in \cite{cfk} and we can give a real form
simply by looking at the fixed points of a given involution on
the $T$-points.
\end{example}

Another point of view on real forms of Lie supergroups is via
the Super Harish-Chandra pairs (SHCP), that
is viewing a supergroup $G$ as a pair $G=(G_0, \fg)$ consisting
of the reduced group $G_0$ and the Lie superalgebra $\fg=\Lie(G)$
(see \cite{ccf} Ch. 7, \cite{vi2} and \cite{cf} for more details).

\begin{definition} \label{shcp-realforms}
Let $(G_0, \fg)$ be a complex analytic SHCP. 
$(r_0, \rho^r)$ is a \textit{real structure} on $(G_0, \fg)$ if
\begin{enumerate}
\item $r_0:G_0 \lra \overline{G_0}$ is a real structure on the ordinary
complex group $G_0$, $G_0^{r}$ being the 
real Lie group of fixed points. 

\item $\rho^r: \fg \lra \fg$ is a $\C$-antilinear 
involutive Lie superalgebra morphism,  $\fg^{r}\simeq\Lie(G_0^r)$
its fixed points.

\item $(r_0, \rho^r)$ are compatible:
 \begin{equation} \label{comp-eq}
\rho^\psi_{\left.\right| \fg_0}  \simeq  
d\psi_0\qquad \Ad(\psi_0(g))\circ\rho^\psi= \rho^\psi\circ\Ad(g) 
\end{equation}
\end{enumerate}
 
$(G_0^r, \fg^r)$
is called a \textit{real form} of $(G_0, \fg)$. 
\end{definition}

This definition is equivalent to
\ref{complexconj}, see \cite{cfk} Sec. 2 for more details.
 
\section{The supergroup $S^{1|1}$ and its
representations}\label{s11-sec}

In \cite{cfk} we have classified all finite dimensional complex 
irreducible representations
not containing the trivial representation of  the real
Lie supergroup $S^{1|1}$. By definition
$S^{1|1}$ is the real form of the complex analytic supergroup $(\C^{1|1})^\times$  
with respect to the involution:
\begin{equation} \label{inv-eq}
\rho: (\C^{1|1})^\times  \lra  
\overline{(\C^{1|1})}^\times, \qquad
\rho(w,\eta)=(\wbar^{-1},   i\wbar^{-2} \etabar)
\end{equation}

We now want to complete the analysis of  \cite{cfk}, and 
compute the representations of $S^{1|1}$, which
correspond to the trivial representation of $S^1$. We recall that  in the 
language of SHCP, in order
to define a representation of $S^{1|1}$, we need to specify the
action of the reduced group $S^1$ and then of $\Lie(S^{1|1})=\langle C,Z
\rangle$
with $[C,C]=[C,Z]=0$, $[Z,Z]=-2C$.

\medskip

Let $W$ be a $1|1$ dimensional super vector space, 
with homogeneous basis $w_0, w_1$.
We define then a representation of $S^{1|1}$ on $W$ by letting $S^1$ 
act trivially on $W$, and letting $Z$ act by:
\begin{equation} \label{W-eq}
Z \cdot w_0 = 0,\qquad
Z \cdot w_1 = w_0.
\end{equation}
Obviously, we may define the parity-reverse of this representation, i.e. 
again letting $S^1$ act trivially on $W$, but this time setting:
\begin{equation} \label{piW-eq}
Z \cdot w_0 = w_1,\qquad
Z \cdot w_1 = 0.
\end{equation}
\noindent 
 The representations $W$ and $\Pi W$ are clearly indecomposable, 
but not semisimple. They are also non-isomorphic, since $\ker(Z)$ is 
$1|0$-dimensional 
in the first case and $0|1$-dimensional in the second one. 
Such representations will play a key role in the classification 
of weight zero representations. One can readily check they are
$\Ad(S^{1|1})$ or $\Pi \Ad(S^{1|1})$.

\begin{proposition}
\label{prop:weight0}
Let $(\pi, \rho, V)$ be a finite-dimensional $S^{1|1}$-representation on which the reduced group acts with weight $0$. Then $(\pi, \rho, V)$ is isomorphic to $U \oplus \bigoplus_i W_i$, where $U$ is a trivial $S^{1|1}$ representation, and each $W_i$ is isomorphic to $\Ad(S^{1|1})$ or $\Pi \Ad(S^{1|1})$.
\end{proposition}

\begin{proof}
As before, the $S^1$-action commutes with the action of $Z$. So let us abuse notation and write $Z$ for $\rho(Z)$. We have $Z^2 = 0$. Let us view $Z$ as an ungraded matrix. Since the characteristic polynomial of $Z$ is quadratic, the Jordan canonical form of $Z$ implies there is a basis of $V$ (not necessarily homogeneous) in which the matrix of $Z$ has the block form
\begin{align*}
\Bigg( \begin{array}{cc}
0 & 0\\
0 & J \\
\end{array}
\Bigg)
\end{align*}
where $J$ is a block-diagonal matrix whose blocks are all of the form
\begin{align*}
J_2 = \begin{pmatrix}
0 & 1\\
0 & 0 
\end{pmatrix}.
\end{align*}
Hence we have a not-necessarily homogeneous basis $w_i, z_j$ of $V$ such that $Z(w_i) = z_i$, $i = 1, \dotsc k$, and $z_j$, $j = 1, \dotsc, l$ are a basis of $\ker(Z)$, for some $k$ and $l$, $k \leq l$. Taking homogeneous components of the $w_i$, we find $Z((w_i)_\alpha) = (z_i)_{\alpha +1}$ for each $i$. (Some of these homogeneous components may be zero, in which case we discard them).

Then each nonzero pair $(w_i)_\alpha, (z_i)_{\alpha + 1}$ is a basis of a $1|1$ dimensional subspace $W_i$ such that $Z((w_i)_\alpha) = (z_i)_{\alpha + 1}$ and $Z((z_i)_{\alpha + 1}) = 0$, i.e. $W_i$ is a representation of the type discussed above
(see (\ref{W-eq}) and  (\ref{piW-eq})). The nonzero homogeneous components of the remaining $z_j$  (those not in the image of $Z$) are a basis for a subspace $U$ of $\ker(Z)$. It is clear that $V = U 
\oplus \bigoplus_i W_i$ and that this splitting is $Z$-invariant, 
whence the theorem.
\end{proof}

\medskip

In order to complete the analysis of the irreducible finite dimensional complex representations of $S^{1|1}$ we recall the following definition from $\cite{cfk}$. For each integer $m\neq 0$,  define the  super weight space $V_m$ to be  a $1|1$ dimensional super vector space, 
with homogeneous basis $v_0, v_1$ on which the reduced group $S^{1}$ acts 
through the character $t \mapsto t^m$, so that infinitesimally
we have 
\begin{align*}
&Z \cdot v_0 = \sqrt{-i m}v_1,\qquad
&Z \cdot v_1 = \sqrt{-i m}v_0.
\end{align*}
It is easily checked that the super weight spaces $V_m$ are irreducible.

Combining the previous proposition with the semisimplicity result 
for super weight spaces with nonzero $m$,
(see \cite{cfk} Theorem 6.1), we have a complete structure theorem for all $S^{1|1}$ representations for which the reduced weight spaces are finite-dimensional, showing that all such representations decompose into a direct sum of super weight spaces and the subspace corresponding to the zero eigenvalue of $C$ whose structure is described in Prop.~\ref{prop:weight0}. 
More precisely, they split into a semisimple part (namely, $\bigoplus_i V_i \oplus U$) and a part which is a direct sum of the indecomposable, non-simple representations $\Ad(S^{1|1})$ and $\Pi \Ad(S^{1|1})$.

\begin{thm}\label{s11-thm}
Let $(\pi, \rho, V)$ be an $S^{1|1}$ representation such that the weight 
spaces $V_m := \{v \in V : t \cdot v = t^mv\}$ for the reduced group $S^1$ 
are all finite-dimensional. Then $(\pi, \rho, V)$ is isomorphic to 
$\bigoplus_i V_i \oplus U \oplus \bigoplus_j W_j$, where the $V_i$ 
are super weight spaces (repeated according to their multiplicities),
$U$ is a trivial representation, and each $W_j$ is isomorphic to 
$\Ad(S^{1|1})$ or $\Pi \Ad(S^{1|1})$. 
\end{thm}

As a corollary we are able to prove the Peter-Weyl theorem
regarding the matrix elements of $S^{1|1}$ representations.
If $\sigma:G \lra \rGL(m|n)$ is a representation of the compact
Lie supergroup $G$, we define the \textit{matrix element}
as the section in $\cO(G)$ corresponding, via the Chart theorem, to the
morphism 
$$
\begin{array}{cccc}
a_{ij}: & G(T) & \mapsto & \C^{1|1}(T) \\
& g & \mapsto & \sigma(g)_{ij}
\end{array}
$$
We have the following result (see \cite{cfk} Theorem 6.3).

\begin{thm}  {\bf The super Peter-Weyl theorem for $S^{1|1}$}.
The complex linear span of the matrix coefficients of the representations
$V_i$'s (see Theorem \ref{s11-thm}) 
and of the adjoint representation
is dense in $\cO(S^{1|1})\otimes \C$.
\end{thm}

\section{The supergroup $\rSU(1|1)$, its
representations and the Peter-Weyl theorem}\label{su-sec}

In this section we study the special unitary supergroup $\rSU(1|1)$
(see \cite{fi} for
its definition and main properties). Since its reduced group is still $S^1$, 
as in the case of $S^{1|1}$ studied in the previous section, we
also may refer to $\rSU(1|1)$ as $S^{1|2}$, that is the supercircle in
two odd dimensions.
We want to view  $\rSU(1|1)$ 
as a real form of the special linear supergroup 
$\rSL(1|1) \cong (\C^{1|2})^\times$
and we shall achieve this, by providing an involution $\sigma$
of $\rSL(1|1)_\R(T)$ functorial in $T \in \smflds_\R$ as we did in
Example \ref{realform-ex11}. 

\medskip
We define, following \cite{fi}, $\sigma: \rSL(1|1)_\R(T)
\lra \rSL(1|1)_\R(T)$
$$
\sigma
\begin{pmatrix} a & \beta \\ \gamma  & d \end{pmatrix} \, = \,
\begin{pmatrix} \dbar^{-1} & -i\abar^{-2}\gabar \\ 
-i\abar^{-2}\betabar  & \abar^{-1}\end{pmatrix}
$$
where $\Ber\begin{pmatrix} a & \beta \\ \gamma  & d \end{pmatrix} =d^{-1}(a-\beta d^{-1} \alpha)=1$.\\
This yields immediately the
supergroup:
$$
\rSU(1|1)(T)=\left\{
\begin{pmatrix} a & \beta \\ -i\betabar a^2 & \abar^{-1} \end{pmatrix}
\, | \, a \abar (1+i \beta \betabar)=1 \right\}
$$
Notice that the relation $a \abar (1+i \beta \betabar)=1$ is effectively
the condition of berezinian equal to $1$.

$\rSU(1|1)$ has dimension $1|2$ and its Lie superalgebra is:
$$
\begin{array}{rl}
\rsu(1|1)&=\left\{
\begin{pmatrix} ix & z \\ -i\zbar & ix \end{pmatrix}\right\} \\ \\
&=\Span\left\{ \, C=\begin{pmatrix} i & 0 \\ 0& i \end{pmatrix}, \,
U=\begin{pmatrix} 0 & 1 \\ -i& 0 \end{pmatrix}, \,
S=\begin{pmatrix} 0 & i \\ -1& 0 \end{pmatrix} \,
\right\}
\end{array}
$$
The commutation relations in $\rsu(1|1)$ are easily seen to be
\begin{align}
\label{eq:commrel}
[C,U]=[C,S]=[U,S]=0, \quad [U,U]=[S,S]=-2C
\end{align}
We refer to \cite{vsv2} pg 112 for more details.
In this work, Varadarajan classifies all 
real forms of $\rsl(m|n)=\Lie(\rSL(m|n))$ and shows 
they are all of the type $\rsu(p,q|r,s)_\pm$.
The $\pm$ refer to the fact, that for each $(p,q|r,s)$, $m=p+q$,
$n=r+s$, we have
two different non isomorphic real forms $\rsu(p,q|r,s)_\pm$, called 
\textit{isomers}, once we consider
isomorphisms fixing a given real form of the even part.
What is surprising is the fact 
that in the physical applications,
these two non isomorphic real forms will give rise to different  
physical fields, which however are associated to the same physics.
Thus for such applications it is irrelevant which form one actually chooses.
For more details see \cite{fl} Ch. 4.

\medskip
As far as we are concerned, we are considering the isomer $\rsu(1|1)_+$
in Varadarajan's notation. As for the isomer $\rsu(1|1)_-$, the corresponding
special unitary supergroup is 
$$
\rSU(1|1)_-(T)=\left\{
\begin{pmatrix} a & i\beta \\ \betabar a^2 & \abar^{-1} \end{pmatrix}
\, | \, a \abar (1-i \beta \betabar)=1 \right\}
$$

\medskip
We are now interested to the theory of representations of such
compact supergroups. For
clarity of exposition we shall discuss just $\rSU(1|1)$, with
Lie superalgebra $\rsu(1|1)=\rsu(1|1)_+$ leaving to
the reader the easy modifications to obtain the representations of
$\rSU(1|1)_-$.
To study such representations we shall resort to the theory
of Super Harish-Chandra pairs (see \cite{ccf}, \cite{cfk}).
We start by noticing that, similarly to the situation of $S^{1|1}$,
(see Sec. \ref{s11-sec} and \cite{cfk}), the reduced
part of $\rSU(1|1)$ is $\rSU(1|1)_{\red}=S^1$, hence its irreducible
representations are all one dimensional and parametrized by
the integers.

In analogy with the $S^{1|1}$ case, we introduce the following super weights spaces. Let $m$ be a nonzero integer (the case $m=0$ being similar 
to the $S^{1|1}$ setting).
Let  $\pi_m^{\pm}$ be the representation of $\rSU(1|1)$   acting in $W^\pm_m\simeq \C^{1|1}$, whose differentiated action is defined
(with a slight abuse of notation) according to
\begin{align*}
C= \left(
\begin{array}{cc}
im & 0\\
0 & im
\end{array}
\right)\, \quad 
U= \left(
\begin{array}{cc}
0 & \sqrt{-im}\\
\sqrt{-im} & 0
\end{array}
\right)\,\quad
S= \left(
\begin{array}{cc}
0 & \frac{\mp m}{\sqrt{-im}}\\
\frac{\pm m}{\sqrt{-im}} & 0
\end{array}
\right)
\end{align*}
\begin{observation}
The representations $\pi^+_m$ and 
$\pi^-_k$ are inequivalent for each nonzero integer $m,k$.
The case $m \neq k$ is obvious, so we only need consider $m = k$. Suppose $F: \pi^+_m \to \pi^-_m$ is an isomorphism of representations. Let $u^+, v^+$ (resp. $u^-, v^-$) be homogeneous bases of $\pi^+_m$ (resp. $\pi^-_m$) such that $C, S, U$ act by the above matrices. We abuse notation by denoting the matrix of $F$ with respect to these bases by $F$. Since the transformation $F$ is even, the matrix $F$ has the form $F=\mathrm{diag}(a,b)$,
for some invertible scalars $a, b$. One sees by direct calculation that the fact that $F$ intertwines the actions of $S$ implies $a = -b$, but the fact that $F$ intertwines the actions of $U$ implies $a = b$, whence $a = b = 0$. This contradicts the assumption that $F$ is an isomorphism.

\end{observation}

\medskip\noindent
We have the following result.

\begin{theorem}\label{su-rep}
Let $(\pi, \rho, V)$ be an $\rSU(1|1)$ representation such that the 
weight spaces $V_m := \{v \in V : t \cdot v = t^mv\}$ for the reduced 
group $S^1$ are all finite-dimensional. Then $(\pi, \rho, V)$ is isomorphic 
to $V_0\bigoplus_m W_m^+ \bigoplus_k W_k^-$, where the $W^\pm_m$ 
are super weight spaces (repeated according to their multiplicities).
\end{theorem}

\begin{proof}
Since $S^1$ is compact we have
\[
V=\bigoplus_{m\in \Z} V_m
\] 
where $m$ are the integers parametrizing the characters 
$\theta\mapsto e^{i m\theta}$. Notice that the isotypic spaces 
$V_m$ are the eigenspaces of the operator $C$.

It follows from the commutation relations~\eqref{eq:commrel} 
that $U$ and $S$ preserve the isotypic decomposition of $V$. 
  Hence, from now on, we can assume 
$V=V_m$, for some nonzero $m\in \Z$. 

We first notice that,  using the commutation relations \eqref{eq:commrel}, 
\[
U^2  =S^2=-im
\]
Hence both $U$ and $S$ are diagonalizable with nonzero eigenvalues. 
Let $w$ be an eigenvector of $U$ with (nonzero) eigenvalue $\lambda$
$(=\pm \sqrt{-im})$.
Then, if $w=w_0+w_1$ denotes its decomposition into homogeneous components, 
we must have
\begin{equation} \label{Urel-eq}
U w_0 =\lambda w_1 \qquad U w_1 =\lambda w_0
\end{equation}
Applying this fact to a basis of eigenvectors of $U$,  $u_j=z_j+\zeta_j$,
$z_j$ even and $\zeta_j$ odd, 
we obtain an homogeneous basis of $V$
\[
z_1,\ldots,z_n|\zeta_1,\ldots \zeta_n
\]
Hence $V$ has dimension $n|n$, for a suitable $n$, corresponding
to the multiplicity of the $V_m$ representation.

We now notice that
\begin{align*}
(US)^2 & = - C^2=m^2 \id
\end{align*}
for some nonzero $m$. Hence $US$ is diagonalizable. Moreover, since it 
is even, we can assume that its eigenvectors are homogeneous. 
Let $\{f_j\}$   denote a basis of such eigenvectors for the even   
part of $V$ with eigenvalues \(\lambda\in\{\pm m\}\). 
We want to show that: 
$U f_j$
is an odd eigenvector for $US$, so that the subspace 
\[
W={\rm span}_\C \{
f_j, Uf_j
\}
\]
is $\rsu(1|1)$-stable.
Since (again from\eqref{eq:commrel}) \( U^2=-im\), we have that 
$U f_j$ is nonzero. Moreover
\[
(US)U f_j= - U (USf_j)
\]
hence $U f_j$ is an odd eigenvector for $US$. The claim that 
the subspace $W= {\rm span}_\C \{f_j, Uf_j
\}$ is invariant follows again from \eqref{eq:commrel} and 
in particular from :
\begin{equation} \label{U-eq}
U^2  =S^2=-im,\qquad 
(US)^2= -C^2
\end{equation}
and 
\[
U^2 S f_j = \lambda U f_j\quad\Leftrightarrow\quad 
-C S f_j =\lambda U f_j \quad\Leftrightarrow\quad -i\,m 
Sf_j =\lambda U f_j 
\]
which gives
\[
Sf_j =\frac{i \lambda}{m} U f_j
\]
Using the basis for $W$:
\[
f_j, \qquad \phi_j=\frac{Uf_j}{\sqrt{-im}},
\]
the explicit action of $\rsu(1|1)$ is obtained as follows.
The action of $C$ is 
\begin{align*}
C(f_j)=imf_j, \qquad C(\phi_j)=im\phi_j
\end{align*}
The matrix representing $C$ in the given basis is then:
\begin{align*}
C= \left(
\begin{array}{cc}
im & 0\\
0 & im
\end{array}
\right)
\end{align*}
The action of $U$ is given by
\begin{align*}
U (f_j) & =\sqrt{-im}\, \frac{Uf_j}{\sqrt{-im}} 
\,=\, \sqrt{-im}\, \phi_j\\
U (\phi_j) & = \frac{U^2 f}{\sqrt{-im}}= 
\frac{-im}{\sqrt{-im}}f_j= \sqrt{-im}f_j
\end{align*}
hence
\[
U = \left(
\begin{array}{cc}
0 & \sqrt{-im}\\
\sqrt{-im} & 0
\end{array}
\right)
\]
The action of $S$ is given by
\begin{align*}
S(f_j) & =\frac{\lambda}{\sqrt{-im}}\,\phi_j \\
S(\phi_j) & =- \frac{ U S f_j}{\sqrt{-im}}= -\frac{\lambda}{\sqrt{-im}}f_j
\end{align*}
hence
\[
S = \left(
\begin{array}{cc}
0 & -\frac{\lambda}{\sqrt{-im}}\\
\frac{\lambda}{\sqrt{-im}} & 0
\end{array}
\right)
\]
with $\lambda \in \{\pm m\}$.
\end{proof}

Next we want to determine the matrix elements for such
representations and prove the Peter-Weyl theorem.

\begin{lemma}\label{g-expr}
We can express uniquely any $T$-point $g \in \rSU(1|1)(T)$ as
$$
g=\mathrm{diag}(t,\tbar^{-1})(1+\theta U)(1+\eta S), \qquad 
$$
where $t,
\tbar \in \cO(T)_0 \otimes \C$, 
$\theta,\eta \in \cO(T)_1 \otimes \C$.
\end{lemma}

\begin{proof}
We need to show that there
are unique $t$, $\theta$, $\eta$, such that
$$
\begin{pmatrix} a & i\beta \\ \betabar a^2 & \abar^{-1} \end{pmatrix}=
\begin{pmatrix} t & 0
\\ 0 & \tbar^{-1}\end{pmatrix}
\begin{pmatrix} 1 & \theta \\ -i\theta & 1\end{pmatrix}
\begin{pmatrix} 1& i\eta \\ -\eta & 1\end{pmatrix}
$$
A direct calculation shows that these are:
$$
t=a(1-\frac{i}{2}\beta\betabar), \quad
\theta= \frac{1}{2}(\betabar a + \beta \abar), \quad
\eta= \frac{1}{2}(\betabar a - \beta \abar).
$$
\end{proof}

\medskip

We finally define the adjoint representation ${\rm Ad}(\rSU(1|1))$ as the representation acting on $\C^{1|2}$ according to 
\[
C=\left(
\begin{array}{ccc}
0 & 0 & 0\\
0 & 0 & 0\\
0 & 0 & 0
\end{array}
\right)
\quad
U=\left(
\begin{array}{ccc}
0 & 1 & 0\\
0 & 0 & 0\\
0 & 0 & 0
\end{array}
\right)
\quad
S=\left(
\begin{array}{ccc}
0 & 0 & 1\\
0 & 0 & 0\\
0 & 0 & 0
\end{array}
\right)
\]

We can now state and prove the Peter-Weyl Theorem for $\rSU(1|1)$.

\begin{theorem} {\bf The super Peter-Weyl theorem for $S^{1|2}$}.
The complex linear span of the matrix coefficients of the 
$\rSU(1|1)$ representations $\{(\pi_m^+)\}_{m\in \Z}$ and of the adjoint representation 
is dense in $\cO(\rSU(1|1))\otimes \C$.
\end{theorem}

\begin{proof}
Let $r_m$ is the irreducible representation described in \ref{su-rep}.
We can view $r_m:\rSU(1|1) \lra \rGL(1|1)$ or alternatively (as
in \ref{su-rep}) as a pair $(r_m^0, \rho_m)$, where $\rho_m$ is a representation
of the Lie superalgebra $\rsu(1|1)$. We
observe that $r_m(I+\theta U)=I+\theta \rho_m(U)$ and
$r_m(I+\eta S)=I+\eta \rho_m(S)$.  Hence we can write using Lemma \ref{g-expr}
for $g \in \rSU(1|1)(T)$:
$$
\begin{array}{rl}
r_m(g)&=r_m\left(\mathrm{diag}(t,\tbar^{-1})(1+\theta U)(1+\eta S)
\right)=\\ \\
&= r_m\begin{pmatrix} t & 0
\\ 0 & \tbar^{-1}\end{pmatrix}\left(I+\theta \rho_m(U)\right)
\left(I+\eta \rho_m(S)\right)=\\ \\
&=\begin{pmatrix} e^{imt} & 0
\\ 0 & e^{imt} \end{pmatrix}\begin{pmatrix} 
1&\sqrt{-im}\theta \\
\sqrt{-im}\theta & 1\end{pmatrix} 
\begin{pmatrix} 1&-i\sqrt{-im}\theta \\
i\sqrt{-im}\theta & 1\end{pmatrix} =\\ \\
&=\begin{pmatrix}e^{imt}(1+m\theta\eta) & e^{imt}\sqrt{-im}(\theta-i\eta) \\
 e^{imt}\sqrt{-im}(\theta+i\eta) & e^{imt}(1-m\theta\eta) \end{pmatrix}
\end{array}
$$
It is then clear that the complex linear span of the matrix coefficients
for $m$ arbitrary and the adjoint representation gives  $\cO(\rSU(1|1))\otimes \C$.

\end{proof}

\end{document}